\documentclass{article}

\usepackage{graphicx}
\usepackage{enumitem}

\usepackage{amsthm}
\theoremstyle{plain}
\newtheorem{theorem}{Theorem}
\newtheorem{proposition}[theorem]{Proposition}

\theoremstyle{definition}
\newtheorem{example}[theorem]{Example}

\usepackage{color}
\usepackage{times}
\usepackage{authblk}
\usepackage{hyperref}
\usepackage{url}
\usepackage{setspace}
\usepackage{amsmath,amssymb}
\usepackage{relsize}

\usepackage{mathtools}

\newcommand{\deltadiam}{\delta_{\mathrm{diam}}}

\newcommand{\whd}{\widehat{\delta}}
\renewcommand{\Re}{\mathbb{R}}

\newcommand{\bx}{\mathbf{x}}

\date{\today}

\title{Submodular and strongly submodular functions and diversities}

\author[1]{David Bryant}
\author[2]{Paul Tupper}

\affil[1]{Department of Mathematics and Statistics, University of Otago\\
PO Box 56, Dunedin 9054, New Zealand\\
\texttt{david.bryant@otago.ac.nz}}

\affil[2]{Department of Mathematics, Simon Fraser University\\
Burnaby, BC, Canada, V5A 1S6\\
\texttt{pft3@sfu.ca}}

\date{\today}

\begin{document}

\maketitle 

\begin{abstract}
Submodular functions and their close relatives play a key role in combinatorial optimization, decision theory and potential theory. Part of their importance and usefulness stems from the connections with convex functions and polytopes. Here we explore connections between these functions and metric theory, with the bridge provided by {\em diversities}, a recently developed generalization of metric spaces to (finite) sets rather than just pairs. Both submodular functions and strongly submodular functions correspond to natural classes of diversities. Submodular diversities, as we define them here, are essentially non-decreasing, intersecting submodular functions which vanish on singletons. 
We prove new geometric embedding results for these diversities. In particular we show that submodular, strongly submodular, and XOS functions can be represented by the {\em generalized circumradius}, a set function in convex analysis equal to the amount a given convex body needs to be stretched to cover a set of points. 
\end{abstract}
 
 \section{Introduction}
  
 A {\em diversity} is a pair $(X,\delta)$ where $X$ is a set and $\delta$ is a real-valued function on finite subsets of $X$ satisfying
 \begin{enumerate}
    \item[(D1)] $\delta(A)=0$ if and only if $|A|\leq 1$;
    \item[(D2)] $\delta(A \cup C) \leq \delta( A \cup B) + \delta(B \cup C)$ if $B$ is nonempty. \footnote{Axioms (D1) and (D2) imply that if $|A| \geq 2$ then $\delta(A) > 0$. This strict inequality causes unnecessary complications, so we will replace (D1) by
\begin{enumerate}
    \item[(D1')] $\delta(A)=0$ if $|A|\leq 1$;
\end{enumerate}
Pairs $(X,\delta)$ satisfying (D1') and (D2), or equivalently, (D1'), (D3), (D4) are called {\em semidiversities}. For conciseness, and following the analogous practice in metric theory, we will use the term {\em diversity} to mean a semidiversity.}
    \end{enumerate}
Written this way, the axioms for diversities look very much like the axioms for a metric space $(X,d)$. Whereas the function $d$ for a metric space describes distances between pairs of objects, $\delta$ describes the spread, range, or diversity, of subsets. Much of the research into diversities has been guided by the idea that diversities are generalized metric spaces \cite{BryantNiesEtal17,BryantNiesEtal21,bryant2024linear,BryantTupper12,BryantTupper14,espinola2014diversities,HaghmaramNourouzi20,Hallback20,jozefiak2023diversity,KirkShahzad14,WuBryantEtal19}. Sometimes the analogies are close, other times the extensions to diversities veer off in new directions. \\

There are alternative sets of axioms for diversities. A pair $(X,\delta)$ is a diversity if and only if it satisfies (D1) together with
 \begin{quote}
\begin{enumerate}
\item[(D3)] $A \subseteq B$ implies $\delta(A) \leq \delta(B)$;
\item[(D4)] $A \cap B \neq \emptyset$ implies $\delta(A \cup B) \leq \delta(A) + \delta(B)$.
\end{enumerate}
\end{quote}
This formulation makes a diversity look a lot less like a generalized metric space and a lot more like an object of combinatorial optimization, perhaps a relative of subadditive or submodular functions.

One of the main goals of this paper is to show how this dual perspective on diversities creates opportunities for arbitrage. We show how results from combinatorial optimization shed light on metric-embedding type problems in diversities, and how geometric representations for diversities provide new perspectives on submodular functions. \\

Let $X$ be a finite set. A function $f:2^X \rightarrow \Re$ is {\em non-decreasing} if $A \subseteq B$ implies $f(A) \leq f(B)$. It is {\em subadditive} if it satisfies
\begin{equation}
f(A \cup B) \leq f(A) + f(B) \label{eq:subadd}
\end{equation}
for all $A,B \subseteq X$. The function $f$ is {\em submodular} if it satisfies 
\begin{equation}
f(A \cup B) + f(A \cap B) \leq f(A) + f(B) \label{eq:submod}
\end{equation} 
for all $A,B \subseteq X$, and {\em strongly submodular} if it satisfies 
\begin{equation}
\sum_{K \subseteq \{1,2,\ldots,m\}} (-1)^{|K|} f\left(A_0 \cup \bigcup_{i \in K} A_i \right) \leq 0\label{def:super_submod}
\end{equation}
for all $m \geq 1$ and $A_0,A_1,\ldots,A_m \in 2^X$. These are all fundamental objects of combinatorial optimization. Submodular functions play the part that convex (and concave!) functions do in continuous optimization \cite{Lovasz1983}, while strongly submodular functions correspond to both  {\em completely alternating} and {\em negative definite} functions in Choquet theory \cite{choquet-1954-capacities,lovasz-2023-submodular-setfunctions} and harmonic analysis \cite{BergChristensenRessel1984}. 

We will encounter another class: {\em XOS functions} \cite{lehmann2006combinatorial}, also known as {\em fractionally subadditive functions} \cite{feige2009maximizing}. These are essentially functions which are pointwise maxima of non-decreasing submodular functions (see below, Proposition~\ref{prop:XOS}). 

Further classes of functions are obtained by only requiring \eqref{eq:subadd}---\eqref{def:super_submod} when $A$ and $B$ intersect. These are the more natural counterparts for diversities: since $\delta$ vanishes on singletons the only {\em fully} submodular diversity is identically zero. We define a function $f:2^X \rightarrow \Re$ to be {\em intersecting subadditive} if it satisfies \eqref{eq:subadd} for all intersecting $A,B \subseteq X$. 
A function $f$ is {\em intersecting submodular} if it satisfies \eqref{eq:submod} for all  $A,B \subseteq X$ which intersect \cite{Frank2011,Fujishige2005,Schrijver2003}. We say that a function $f:2^X \rightarrow \Re$  is {\em strongly intersecting submodular} if it satisfies \eqref{def:super_submod} for all $m \geq 1$ and $A_0,A_1,\ldots,A_m \in 2^X$ such that $A_0 \neq \emptyset$. 

These concepts map directly over to diversities. A pair $(X,\delta)$ such that $\delta(A) = 0$ when $|A| \leq 1$ is a diversity if and only if  $\delta$ is non-decreasing and intersecting subadditive. We say that a diversity $(X,\delta)$ is {\em submodular} if $\delta$ is intersecting submodular. We will show that diversities with $\delta$ equal to a strongly intersecting submodular are exactly the {\em diversities of negative type}, as introduced by \cite{WuBryantEtal19} and discussed below. \\~\\

We will see that, via diversities, we obtain new geometric representations of many of these classical combinatorial optimization functions. Motivated by applications of metric embeddings to hard combinatorial optimization problems, Bryant and Tupper \cite{bryant2024linear} investigated embeddings into diversities on $\Re^k$.  A diversity $(\Re^k,\delta)$ is {\em linear} if $\delta$ satisfies $\delta(\lambda A) = \lambda \delta(A)$ and $\delta(A+B) = \delta(A) + \delta(B)$ for all $\lambda \geq 0$ and nonempty, finite $A,B \subseteq \Re^k$. It is {\em sublinear} if $\delta$ satisfies $\delta(\lambda A) = \lambda \delta(A)$ and $\delta(A+B) \leq \delta(A) + \delta(B)$ for all $\lambda \geq 0$ and nonempty, finite $A,B \subseteq \Re^k$. 

Many important examples of diversities on $\Re^k$ are linear or sublinear \cite{bryant2024linear}. One canonical example is the {\em generalized circumradius}. Let $K$ be a compact convex subset of $\Re^k$ with nonempty interior. The generalized circumradius $R(A,K)$ of a bounded set $A$ with respect to {\em kernel} $K$ is defined by 
\[R(A,K) = \inf \{\lambda \geq 0 : A \subseteq \lambda K + z \mbox{ for some $z \in \Re^k$ } \}.\]
That is, $R(A,K)$ is the minimum amount that we have to scale $K$ so that a translate covers $A$ (Fig.~\ref{fig:Minkowski}). When $K$ is the unit ball, $R(A,K)$ is the circumradius of $A$. The restriction of $R(A,K)$ to finite subsets $A$ is a diversity, called the Minkowski diversity, denoted $(\Re^k,\delta_K)$ \cite{bryant2023diversities}. 
\begin{figure}[htb]
\centerline{\includegraphics[width=0.7\textwidth]{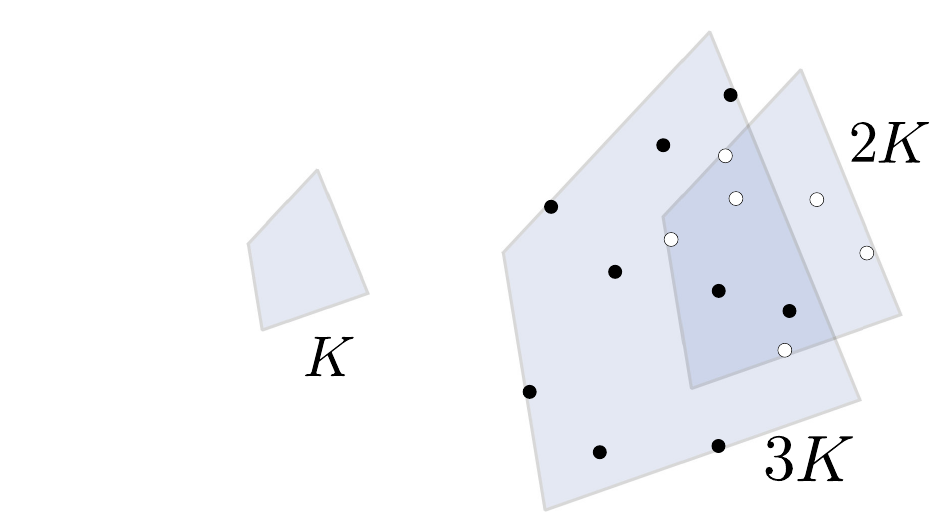}}
\caption{\label{fig:Minkowski} \small An illustration of generalized circumradius. Here $K$ is the kernel, $A$ is the set of filled dots and $B$ the set of hollow dots. In this example, $R(A, K) = 3$ and $R(B,K) = 2$. The set $A$ can be covered by a translate of $3K$ but not by any translate of $\lambda K$ for $\lambda < 3$. Likewise,  $B$ can be covered by a translate of $2K$ but not by any translate of $\lambda K$ for $\lambda < 2$.}
\end{figure}

Bryant and Tupper \cite{bryant2024linear} investigated embeddings into sublinear and linear diversities. An (isometric) embedding of a diversity $(X,\delta_X)$ into a diversity $(Y,\delta_Y)$ is a map $\phi:X \rightarrow Y$ such that $\delta_Y(\phi(A)) = \delta_X(A)$ for all finite $A \subseteq X$. A diversity is  {\em linear (sublinear)-embeddable} if there is an isometric embedding of $(X,\delta)$ into a linear (sublinear) diversity $(\Re^k,\delta)$ for some $k$. It is {\em Minkowski-embeddable} if it can be embedded into $(\Re^k,\delta_K)$ for some $k \geq 1$ and kernel $K \subseteq \Re^k$. We typically constrain $(X,\delta_X)$ to be a finite diversity, that is, one for which $|X|$ is finite. 

There are elegant characterizations of linear and sublinear embeddability which are, in many senses, analogous to Menger's and Schoenberg's characterizations for Euclidean embeddings of metric spaces. 

\begin{theorem} \label{thm:linearChar} (Theorem~9 in \cite{bryant2024linear})
Let $(X,\delta)$ be a finite diversity. The following are equivalent:
\begin{enumerate}
\item $(X,\delta)$ is linear-embeddable.
\item There is an isometric embedding for some $k$ from $(X,\delta)$ into a Minkowski diversity $(\Re^k,\delta_K)$, where $K$ is a $k$-dimensional simplex.
\item $(X,\delta)$ has negative type: for all zero-sum vectors $x$ indexed by nonempty subsets of $X$ we have
\[\sum_{A,B \neq \emptyset} x_A x_B \delta(A \cup B) \leq 0.\]
\end{enumerate}
\end{theorem}

\begin{theorem} \label{thm:sublinearChar} (Theorem~10 in \cite{bryant2024linear} ) 
Let $(X,\delta)$ be a finite diversity. The following are equivalent:
\begin{enumerate}
\item $(X,\delta)$ is sublinear-embeddable.
\item There is an isometric embedding for some $k$ from $(X,\delta)$ into a Minkowski diversity $(\Re^k,\delta_K)$, for some convex compact subset $K$ of $\Re^k$ with nonempty interior.
\item There is a finite collection $(X,\delta_1)$, $(X,\delta_2),\ldots,(X,\delta_N)$ of negative type diversities such that 
\[ \delta(A) = \max\{\delta_1(A),\delta_2(A),\ldots,\delta_N(A)\}\]
for all $A \subseteq X$.
\end{enumerate}
\end{theorem}
~\\

\subsection*{Outline}

We outline the structure of the paper and the main results.

Section~2 begins with a concise introduction to diversities and diversity theory, before reviewing results on submodular functions, strongly submodular functions and diversities of negative type. 

Section~3 focuses on the relationship between strongly submodular functions and diversities of negative type. We show that diversities $(X,\delta)$ of negative type are exactly those for which the functions $\delta_x:2^{X \setminus \{x\}} \rightarrow \Re$ are strongly submodular for all $x$, a connection which leads to further characterizations. We present two ways to construct a diversity of negative type from a strongly submodular function. 

Section~4 introduces the class of {\em submodular diversities}, which has the same relationship to submodular functions as diversities of negative type have to strongly submodular functions. We show that the construction in Section~3 of a diversity from a strongly submodular function extends to the submodular case. The main result of this section, and of the paper, is Theorem~\ref{thm:embed_submodular}, which shows that submodular diversities are sublinear-embeddable. Most of the work in proving this result is in Theorem~\ref{thm:submodular_match}. 

We note that the proofs of Theorem~\ref{thm:submodular_match} and Theorem~\ref{thm:embed_submodular} are not based on the standard polymatroid constructions of submodular function theory. They explore a different polytope: that formed from the set of negative type diversities dominated by a given submodular diversity. The proof of sublinearity is based on the fact that, for each subset, we can write $\delta(A)$ as the maximum of $\eta(A)$ for diversities in this polytope, analogous to the relationship between submodular functions and additive functions in the corresponding polymatroid. This construction may well be useful when proving other results about intersecting submodular functions and their close relatives. 

Section~5 builds on the embedding results in Section~4 but gives a complete characterization: a diversity is sublinear-embeddable if and only if it is the maximum of submodular diversities. This moves the focus from submodular functions to XOS functions, though we show that the attractive pairing of classes of diversities and classes of set functions does not quite extend to XOS functions. 

Section~6 takes the results about diversities and their embeddings and maps them back to results on submodular functions and the associated classes. We connect submodular functions to the {\em generalized circumradius}, a geometric connection which we believe is quite new. Strongly submodular functions provide a special case, while XOS functions (that is, maxima of submodular functions) are exactly those for which this representation is possible.

\section{Background}

\subsection{A brief introduction to diversities and diversity theory}

The most natural way to think about diversities has been that they are a generalization of metric spaces where we assign values to (finite) subsets instead of just pairs. One of the simplest but most fundamental diversities is {\em diameter} in a metric space $(X,d)$:
\[ \deltadiam(A) = \max\{d(a,b):a,b  \in A \},\]
for nonempty $A$ and $\deltadiam(\emptyset)=0$.
We see that $\deltadiam$ vanishes exactly on singletons and the empty set, it is non-decreasing, and (D4) follows by a straightforward case-by-case analysis.  The diameter diversity $(X,\deltadiam)$  is extremal: if $(X,\delta)$ is any other diversity with induced metric $(X,d)$ then  $\delta(A) \geq \deltadiam(A)$ for all finite $A \subseteq X$. 

Historically, the concept of diversities grew out of tight span theory, particularly the connection with phylogenetics. Consider a phylogenetic tree with branch lengths (representing, for example, expected number of substitutions per site) and some vertices labelled by elements of a set of {\em taxa} $X$. Given $x,y \in X$ the {\em additive distance} $d_T(x,y)$ is the sum of branch lengths along the unique path in the tree connecting them. Dress \cite{dress1984trees} showed that the tree (as a continuous space) could be reconstructed from $(X,d_T)$ using the {\em tight span}. 

Additive distances extend naturally to a diversity. For a subset $A \subseteq X$ we define the {\em phylogenetic diversity} $(X,\delta_T)$ whereby $\delta_T(A)$ is the sum of the branch lengths for branches in the smallest connected subtree containing $A$, see Figure~\ref{fig:phylo}. The concept of phylogenetic diversity was introduced by \cite{Faith92}. Bryant and Tupper extended the definition of the tight span to diversities, showing that (as with additive distances) the tight span of a phylogenetic diversity was a continuous space given by the underlying phylogeny \cite{BryantTupper12}. 

\begin{figure}[ht]
\begin{center}
\includegraphics[width=0.8\textwidth]{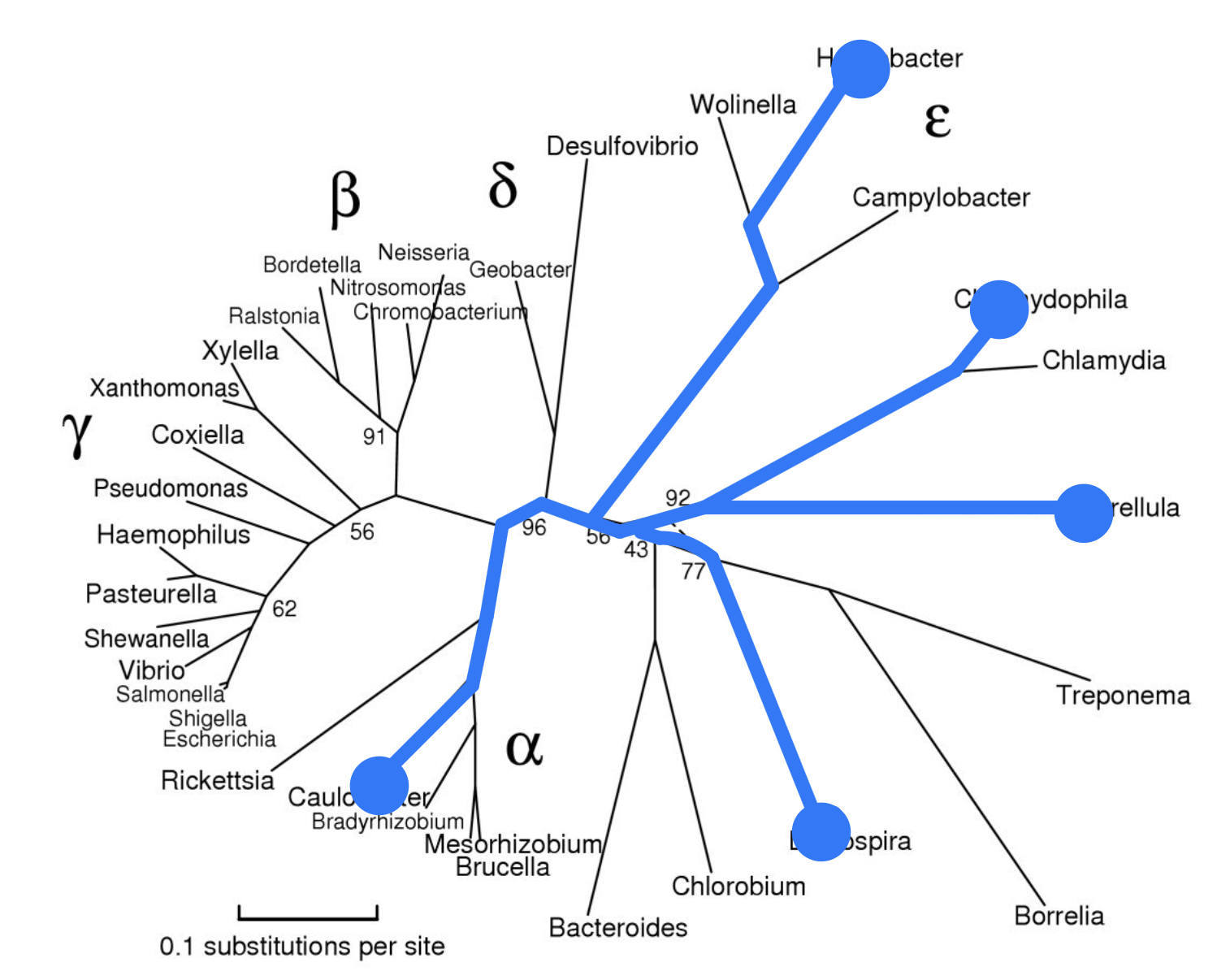}
\end{center}
\caption{\label{fig:phylo} \small An illustration of phylogenetic diversity, as introduced by \cite{Faith92}. The phylogenetic diversity of a set of taxa (in this case a group of bacteria) is defined to be the sum of the (highlighted) branch lengths on the smallest subtree connecting them.}
\end{figure}

Additive distances are textbook examples of $\ell_1$ embeddable metrics, meaning that there is an isometry from $(X,d_T)$ to $(\Re^k,d_1)$ for some $k \geq 1$, where
\[ d_1(x,y) = \sum_{i=1}^k |x_i - y_i|.\]
The diversity analogue of the $\ell_1$ metric is the $\ell_1$ diversity $(\Re^k,\delta_1)$ given by
\[ \delta_1(A) = \sum_{i=1}^k \max\{|a_i - b_i|: a,b \in A\}.\]
Every phylogenetic diversity can be isometrically embedded into an $\ell_1$ diversity, though $\ell_1$-embeddable diversities are much more general. Indeed the whole machinery of $\ell_1$-metric embeddings extends to diversities  \cite{BryantTupper14}, at least in principle. \\

Other examples of diversities arise in classical combinatorial optimization problems, including the length of a minimum-length travelling-salesperson tour through a set of points, the length of the optimal Steiner tree connecting a set of points, and the length (sum of generator lengths) of a minimal zonotope containing a set of points, see \cite{BryantTupper12,BryantTupper14,bryant2023diversities,bryant2024linear}. \\

Convex analysis provides other examples of diversities. The {\em Minkowski sum} of two sets $A,B$ in a vector space is defined by $A+B =\{a+b: a \in  A,\, b\in B\}$. The Minkowski sum satisfies
\[(A \cup B) + (A \cap B) \subseteq A + B\]
for all $A,B$. For $\lambda \in \Re$ we write $\lambda A = \{\lambda a: a \in A\}$. A function $f$ defined on nonempty subsets of $\Re^k$  is {\em Minkowski linear} if it satisfies 
\begin{align}
f( \lambda A) & = \lambda f(A)  \label{eq:fhomog} \\
f(A + B) & = f(A) + f(B) \label{eq:fadd}
\end{align}
for all $\lambda \geq 0$ and nonempty $A,B \subseteq \Re^k$. We say that $f$ is {\em Minkowski sublinear} if it satisfies \eqref{eq:fhomog} together with 
\begin{align}
f(A + B) & \leq f(A) + f(B) \label{eq:fsubadd}
\end{align}
for all nonempty $A,B \subseteq \Re^k$. Linear and sublinear functions appear frequently in the Minkowski theory of valuations (see, e.g., \cite{Schneider14}).  We say that a diversity $(\Re^k,\delta)$ is linear, as defined above, if $\delta$ is Minkowski linear; it is sublinear if $\delta$ is Minkowski sublinear. \\

The $\ell_1$ diversity $(\Re^k,\delta_1)$ is linear \cite{bryant2024linear}. Another linear diversity appearing in convex analysis is {\em mean width}. Let $u$ be a direction vector of unit length. The {\em width} of a set $A$ in direction $u$ is then given by 
\[\max\{u^Ta:a \in A\} - \min\{u^Tb: b \in A\} = \max\{u^T(a-b): a,b \in A\} .\]
The {\em mean width} of $A$ is the average or expected width of $A$ if we pick $u$ uniformly at random. Formally, the mean width diversity is defined by $\delta_w(\emptyset) = 0$ and
\[\delta_w(A) = 
\frac{2}{\omega_k} \int_{\mathbb{S}^{k-1}} \sup\{a^T x : a \in A \}  \, \mathrm{d} \nu(x) \]
for finite, nonempty $A \subseteq \Re^k$. Here, $\nu$ is the (uniform) rotation-invariant measure on the sphere and the constant 
$\omega_k$ is chosen so that $\delta_w(\{a,b\}) = \|a-b\|_2$ for all $a,b \in \Re^k$.\\

The diameter diversity $(\Re^k,\deltadiam)$ is sublinear. The Minkowski diversity 
\[\delta_K(A) = \inf\{\lambda \geq 0: A \subseteq \lambda K + z \mbox{ for some $z \in \Re^k$}\},\]
defined above, is a canonical example of a sublinear diversity. Other examples of sublinear diversities are discussed in \cite{bryant2024linear}. \\

To summarize, we have a nested chain of diversity classes:
\[\mbox{ Phylogenetic diversities}  \subsetneq \mbox{$\ell_1$-embeddable} \subsetneq \mbox{linear-embeddable}  \subsetneq \mbox{sublinear-embeddable.} \]
By Theorem~\ref{thm:linearChar}, linear-embeddable diversities are exactly the negative-type diversities, which are exactly the diversities which are Minkowski-embeddable with a simplex kernel $K$. By Theorem~\ref{thm:sublinearChar} finite sublinear-embeddable diversities are exactly the Minkowski-embeddable diversities. We will be fitting submodular diversities into this hierarchy.

\subsection{Properties of submodular and strongly submodular functions}

Recall that a function $f:2^X \rightarrow \Re$ is submodular if it satisfies \eqref{eq:submod} for  all $A,B \subseteq X$. Here we collect facts about submodular functions, all of which can be found in  \cite{lovasz-2023-submodular-setfunctions}, see also reviews in \cite{Frank2011,Fujishige2005,Lovasz1983,Schrijver2003}. 

\begin{proposition} \label{prop:basic_submodular}
\begin{enumerate}
\item If $f_1,f_2$ are submodular and $\alpha \geq 0$ then $\alpha f_1$ and $f_1 + f_2$ are submodular. 
\item If $f:2^X \rightarrow \Re$ is submodular and non-decreasing and $\tau \in \Re$ then the truncation $\min\{f(A),\tau\}$ is also submodular. 
\item A function $f$ is non-decreasing and submodular  if and only if
\begin{equation}
f(A_0) - f(A_0 \cup A_1) - f(A_0 \cup A_2) + f(A_0 \cup A_1 \cup A_2) \leq 0 \label{eq:submod2}
\end{equation}
for all $A_0,A_1,A_2 \in 2^X$. 
\item If $\Gamma:2^X \rightarrow 2^Y$ satisfies $\Gamma(A \cup B) = \Gamma(A) \cup \Gamma(B)$ for all $A,B \in 2^X$ and $g:2^Y \rightarrow \Re$ is non-decreasing and submodular then $f:2^X \rightarrow \Re$ given by $f(A) = g(\Gamma(A))$ for all $A \subseteq X$ is non-decreasing and submodular.  
\end{enumerate}
\end{proposition}

Strongly submodular functions are the subclass of submodular functions which satisfy 
\begin{equation}
\sum_{K \subseteq \{1,2,\ldots,m\}} (-1)^{|K|} f\left(A_0 \cup \bigcup_{i \in K} A_i \right) \leq 0
\end{equation}
for all $m \geq 1$ and $A_0,A_1,\ldots,A_m \in 2^X$. They were named {\em strongly submodular} by Lov\'asz \cite{lovasz-2023-submodular-setfunctions}; Choquet \cite[Chap.~3]{choquet-1954-capacities} called them {\em alternating functions}, albeit with equivalent but slightly different notation. These functions also arise in \cite{BergChristensenRessel1984} defined on the semigroup $(2^X,\cup)$, where it is shown that, on this particular semigroup, they are exactly the {\em negative definite functions}. 

Again we collect results from \cite{lovasz-2023-submodular-setfunctions} for later use. 

\begin{proposition} \label{prop:basic_strong}
Let $X$ be a finite set.
\begin{enumerate}[label=(\roman*)]
\item For all $\alpha \in \Re$ the constant function $f(A) = \alpha$ is strongly submodular.
\item If $f_1,f_2$ are strongly submodular and $\alpha \geq 0$ then $\alpha f_1$ and $f_1 + f_2$ are strongly submodular.
\item Every strongly submodular function $f:2^X \rightarrow \Re$  is submodular.
\item Every strongly submodular function $f:2^X \rightarrow \Re$ is non-decreasing. 
\item If $\Gamma:2^X \rightarrow 2^Y$ satisfies $\Gamma(A \cup B) = \Gamma(A) \cup \Gamma(B)$ for all $A,B \in 2^X$ and $g:2^Y \rightarrow \Re$ is strongly submodular then $f:2^X \rightarrow \Re$ given by $f(A) = g(\Gamma(A))$ for all $A \subseteq X$ is strongly submodular. 
\item If $w_x \geq 0$ for all $x \in X$ then the function $f(A) = \sum_{a \in A} w_a$ is strongly submodular. 
\item Suppose $f:2^X \rightarrow \Re$ is zero on the empty set. Then $f$ is strongly submodular if and only if 
\begin{equation}
\sum_{B:A \subseteq B} (-1)^{|B \setminus A|} f(B) \leq 0 \label{eq:strong_moebius}
\end{equation}
for all $A \subsetneq X$.  
\item If $f(\emptyset) = 0$ then $f$ is strongly submodular if and only if there are subsets $A_1,\ldots,A_m$ and non-negative weights $w_1,\ldots,w_m$ such that $f(A) = \sum_{i:A \cap A_i \neq \emptyset} w_i$. 
\item If $w_x \geq 0$ for all $x \in X$ then the function $f:2^X \rightarrow \Re$ with $f(\emptyset) = 0$ and $f(A) = \max\{w_a:a \in A\}$ for nonempty $A \subseteq X$ is strongly submodular. 
\item If $\omega_x \in \Re$ for all $x \in X$ then the function $g:2^X \rightarrow \Re$ with $g(A) = \max\{\omega_a:a \in A\}$ for nonempty $A \subseteq X$ and 
$g(\emptyset) \leq  \min\{\omega_a:a \in X\}$  is strongly submodular. 
\end{enumerate}
\end{proposition}
\begin{proof}
(i). to (vii). are all found in Section~9 of \cite{lovasz-2023-submodular-setfunctions}. \\
(viii). corresponds to Theorem~9.4(ii) of \cite{lovasz-2023-submodular-setfunctions}. This shows that strongly submodular functions with $f(\emptyset) = 0$ are exactly the {\em coverage functions} \cite{Badanidiyuru2012Sketching}. \\
For (ix), let $W = [0,\infty)$ with Lebesgue measure $\mu$ and define the relation $\mathcal{R} = \{(a,t) \in X \times W: t \leq w_a\}$. Defining $\mathcal{R}(A) = \{t \in W: (a,t) \in R \mbox{ for some $a \in A$}\}$ we have $\mathcal{R}(\emptyset) = \emptyset$ and $\mathcal{R}(A) = \bigcup_{a \in A} [0,w_a] = [0,f(A)]$ so $f(A) = \mu(\mathcal{R}(A))$ for all $A$ and $f$ is strongly submodular by Theorem~9.4(iii) of \cite{lovasz-2023-submodular-setfunctions}. \\
For (x), define $w_a = \omega_a - g(\emptyset)$ for all $a \in X$ and define $f$ as in (ix). As $g$ differs from $f$ by a constant and $f$ is strongly submodular, so is $g$. 
\end{proof}

\subsection{Properties of negative-type diversities}

Diversities of negative type were introduced by \cite{WuBryantEtal19} as an analogue of metrics of negative type, though the same concept pops up in multiple contexts. When $X$ is finite, a pair $(X,\delta)$ is a diversity of negative type if 
\[ \sum_{A,B \subseteq X} \bx_A \bx_B \delta(A \cup B) \leq 0\]
for all vectors $\bx$ such that $\bx_\emptyset = 0$ and $\sum_{A \neq \emptyset} \bx_A = 0$. Notably, every $\ell_1$-embeddable diversity is negative type, though the converse is not true \cite{WuBryantEtal19}. Diversities of negative type correspond to negative definite functions on the semigroup $(2^X \setminus  \{\emptyset\}, \cup)$ \cite{BergChristensenRessel1984}. Note that we exclude the empty set from this semigroup to line up the definitions. 

The most useful tool for working with diversities of negative type is the following M\"obius transform formulation. Given a finite set $X$ and any function $f:2^X \rightarrow \Re$, define 
\[\lambda[f]_A = \sum_{B:A \subseteq B} (-1)^{|A|+|B|+1} f(B)\]
for all $A \subseteq X$. This equation has the inverse
\[f(A) = -\sum_{B:A \subseteq B} \lambda[f]_B,\]
see \cite{WuBryantEtal19}. The similarity to \eqref{eq:strong_moebius} is no coincidence. 

\begin{proposition} \label{prop:charNegBonus}
Let $\delta:2^X \rightarrow \Re$ be a function defined on subsets of a finite set $X$. Then $(X,\delta)$ is a diversity of negative type if and only if
$\lambda[\delta]_A \geq 0$ for all $A \neq \emptyset, X$,
\begin{equation} \label{eq:negTypeEmpty}
    \sum_{A \subseteq X} \lambda[\delta]_A = 0
\end{equation}
and
\begin{equation} \label{eq:negTypeSingleton}
\sum_{A \subseteq X:x \in A} \lambda[\delta]_A = 0
\end{equation}
for all $x$. 
\end{proposition}
\begin{proof}
Suppose that $(X,\delta)$ is a diversity of negative type. Then $\lambda_A \geq 0$ for all $A \neq \emptyset,X$ by Theorem~2 of \cite{WuBryantEtal19}. We also have
\[0 = \delta(\emptyset) = -\sum_{A \subseteq X} \lambda_A\]
and for all $x \in X$,
\[0 = \delta(\{x\}) = -\sum_{A \subseteq X,\, x \in A} \lambda_A.\]

For the converse, suppose that $\lambda[\delta]_A \geq 0$ for all $A \neq \emptyset,X$ and that $\lambda[\delta]$ satisfies \eqref{eq:negTypeEmpty} and \eqref{eq:negTypeSingleton} for all $x \in X$. By \eqref{eq:negTypeEmpty} and \eqref{eq:negTypeSingleton} we have $\delta(\emptyset) = 0$ and $\delta(\{x\}) = 0$ for all $x \in X$. The fact that $(X,\delta)$ is a negative-type diversity now follows by Theorem~2 of \cite{WuBryantEtal19}.
\end{proof}

\section{Strongly submodular functions and diversities of negative type}

We start by showing that diversities of negative type are exactly the diversity equivalent of strongly submodular functions. 

\begin{proposition} \label{prop:neg_is_strongsubmod}
Let $X$ be a finite set and let $\delta$ be a function on subsets of $X$ such that $\delta(\emptyset) = 0$. The following are equivalent:
\begin{enumerate}[label=(\roman*)]
\item $(X,\delta)$ is a diversity of negative type;
\item For all $x \in X$ the function $\delta_x:2^{X \setminus\{x\}} \rightarrow \Re$ is strongly submodular and zero on the empty set;
\item $\delta$ is zero on singletons and is strongly intersecting submodular;
\item For all $x \in X$ there is an $\ell_1$-embeddable diversity $\delta_1^{(x)}$ such that $\delta(A) = \delta_1^{(x)}(A)$ for all $A \subseteq X$ such that $x \in A$.
\end{enumerate}
\end{proposition}
\begin{proof}
(i) $\Leftrightarrow$ (ii).\\
For all $x \in X$ and $Y \subsetneq X \setminus \{x\}$, 
\begin{align*}
\sum_{Z \subseteq X \setminus \{x\}:Y \subseteq Z} (-1)^{|Z \setminus Y|} \delta_x(Z)  & = \sum_{Z \subseteq X \setminus \{x\}:Y \subseteq Z} (-1)^{|Z \setminus Y|} \delta(Z \cup \{x\}) \\
& =   \sum_{Z \subseteq X \setminus \{x\} :Y \subseteq Z} (-1)^{|Z| - |Y|} \delta(Z \cup \{x\})  \\
& =  \sum_{Z \subseteq X: Y \cup \{x\} \subseteq Z} (-1)^{|Y| + |Z| + 1} \delta(Z) \\
& = -\lambda[\delta]_{Y \cup \{x\}}.
\end{align*}
The result now follows from Propositions~\ref{prop:basic_strong} (vii) and \ref{prop:charNegBonus}. \\
(ii) $\Rightarrow$ (iii).\\
For all $x \in X$, $\delta(\{x\}) = \delta_x(\emptyset) = 0$, so $\delta$ is zero on singletons. Suppose that $A_0,A_1,\ldots,A_m$ are subsets of $X$ and $A_0 \neq \emptyset$. Choose any $x \in A_0$ and define 
$B_i = A_i \setminus \{x\}$ for all $i=0,\ldots,m$. 
Then 
\[\sum_{K \subseteq \{1,2,\ldots,m\}} (-1)^{|K|} \delta \left(A_0 \cup \bigcup_{i \in K} A_i \right) = \sum_{K \subseteq \{1,2,\ldots,m\}} (-1)^{|K|} \delta_x \left(B_0 \cup \bigcup_{i \in K} B_i \right)  \leq 0\]
as $\delta_x$ is strongly submodular. Hence $\delta$ is strongly intersecting submodular.\\
(iii) $\Rightarrow$ (ii) \\
Fix $x \in X$. Then $\delta_x(\emptyset) = \delta(\{x\}) = 0$ as $\delta$ is zero on singletons. Given subsets $A_0,A_1,\ldots,A_m$ of $X \setminus \{x\}$ define $B_i = A_i \cup \{x\}$ for $i=0,\ldots,m$. Then 
\[\sum_{K \subseteq \{1,2,\ldots,m\}} (-1)^{|K|} \delta_x\left(A_0 \cup \bigcup_{i \in K} A_i \right) = \sum_{K \subseteq \{1,2,\ldots,m\}} (-1)^{|K|} \delta\left(B_0 \cup \bigcup_{i \in K} B_i \right) \leq 0,\]
so $\delta_x$ is strongly submodular by \eqref{def:super_submod}. \\
(i) $\Rightarrow$ (iv).\\
Fix $x \in X$. For all nonempty $A \subseteq X$ let
\[\alpha_A = \begin{cases} \lambda[\delta]_A & \mbox{ if $x \in A$ } \\ \lambda[\delta]_{X \setminus A} & \mbox{ if $x \not \in A$ } \end{cases}\]
and $\alpha_\emptyset = -\sum_{A \neq \emptyset} \alpha_A$. 
Define $\delta_1^{(x)}$ by \[\delta_1^{(x)} (A) = -\sum_{B:A \subseteq B} \alpha_B\]
for all $A \subseteq X$. Then $(X,\delta_1^{(x)})$ is an $\ell_1$-embeddable diversity by Proposition~8 of \cite{WuBryantEtal19}. 
(iv) $\Rightarrow$  (i)\\
By Proposition~3 of \cite{WuBryantEtal19}, any $\ell_1$-embeddable diversity $\delta^{(x)}_1$ has negative type. Hence for all nonempty $A \subsetneq X$ and $x \in A$ we have 
\[\lambda[\delta]_A = \sum_{B:A \subseteq B}(-1)^{|A|+|B|+1} \delta(B) = \sum_{B:A \subseteq B}(-1)^{|A|+|B|+1} \delta_1^{(x)}(B) = \lambda[\delta^{(x)}_1]_A \]
which is non-negative by Proposition~\ref{prop:charNegBonus}. By the same proposition, $(X,\delta)$ has negative type. \\

\end{proof}

We describe two ways to construct a diversity of negative type from an arbitrary strongly submodular function.

\begin{proposition} \label{prop:extend_strong}
\begin{enumerate}
\item Let $f:2^X \rightarrow \Re$ be a strongly submodular function. Then $(X,\delta)$ given by $\delta(\emptyset) = 0$ and 
\[\delta(A) = f(A) - \min\{f(\{a\}): a \in A \} \quad \quad (A \neq \emptyset)\]
is a diversity of negative type.
\item Fix $x \in X$. Let $g:2^{X \setminus \{x\}} \rightarrow \Re$ be a strongly submodular function with $g(\emptyset) = 0$. Then $(X,\whd)$ given by 
\[
\whd(A) = \begin{cases} g(A \setminus \{x\}) & \mbox{ if $x \in A$;} \\
g(A) - \min\{g(\{a\}):a \in A\} & \mbox{ if $A \neq \emptyset$ and $x \not \in A$;} \\
0 & \mbox{ if $A = \emptyset$,} \end{cases}
\]
has negative type.
\end{enumerate}
\end{proposition}
\begin{proof}
1.  Define $\rho:2^X \rightarrow \Re$ by $\rho(A) = \max\{-f(\{a\}): a \in A\}$ for all nonempty $A \subseteq X$, with $\rho(\emptyset) = \min\{-f(\{a\}):a \in  X\}$. Then $\rho$ is strongly submodular, by Proposition~\ref{prop:basic_strong}(x), and hence so is $h = f + \rho$. As $\delta(A) = h(A)$ for all $A \neq \emptyset$ we have that $\delta$  satisfies \eqref{def:super_submod} for all $m \geq 1$ and $A_0,A_1,\ldots,A_m \in 2^X$ such that $A_0 \neq \emptyset$. Hence $\delta$ is strongly intersecting submodular, and as $\delta(A)  = 0$ when $|A| \leq 1$ we have from Proposition~\ref{prop:neg_is_strongsubmod} that $(X,\delta)$ has negative type.  \\
2. Define $f:2^X \rightarrow \Re$ by $f(A) = g(A \setminus \{x\})$ for all $A \subseteq X$. Then $f$ is strongly submodular by Proposition~\ref{prop:basic_strong} (v), $f(\{a\}) \geq 0$ for all $a \in X$, $f(\{x\}) = 0$, so 
\[\whd(A) = f(A) - \min\{f(\{a\}): a \in A \}\]
for all nonempty $A \subseteq X$. The result follows by applying the construction in part 1.
\end{proof}

\section{Submodular functions and diversities}

Recall that a diversity $(X,\delta)$ is {\em submodular} if $\delta$ is intersecting submodular, which holds when 
\[\delta(A \cup B) + \delta(A \cap B) \leq \delta(A) + \delta(B)\]
whenever $A \cap B \neq \emptyset$.

. 

\begin{proposition} \label{prop:basic_submodular_diversities}
\begin{enumerate}
\item A diversity $(X,\delta)$ is submodular if and only if for all $x \in X$ the function $\delta_x:2^{X \setminus \{x\}} \rightarrow \Re$ given by $\delta_x(A) = \delta(A \cup \{x\})$ is submodular.
\item Every diversity of negative type is submodular. 
\end{enumerate}
\end{proposition}
\begin{proof}
1. follows directly from the definition of intersecting submodular functions while 2. is a consequence of Proposition~\ref{prop:neg_is_strongsubmod} and Proposition~\ref{prop:basic_strong}(iii).
\end{proof}

The inclusion of negative type diversities within the class of submodular diversities is strict.

\begin{example} \label{ex:submod_not_linear}
Let $X = \{a,b,c,d\}$. The diversity $(X,\delta)$ with $\delta$ given by
\[\delta(A) = \begin{cases} 2 & \mbox{if $|A| \geq 3$;}\\ 1 & \mbox{if $|A|=2$} \\ 0 & \mbox{ otherwise} \end{cases} \]
is submodular but not linear-embeddable.
\end{example}
\begin{proof}
We have 
\begin{align*}
\lambda[\delta]_{\{a\}} & = -\delta(\{a\}) + \delta(\{a,b\}) + \delta(\{a,c\}) + \delta(\{a,d\}) - \delta(\{a,b,c\}) - \delta(\{a,b,d\}) - \delta(\{a,c,d\}) + \delta(\{a,b,c,d\}) \\
& = -1
\end{align*}
so by Proposition~\ref{prop:charNegBonus}, $(X,\delta)$ is not negative-type and hence, by Theorem~\ref{thm:linearChar}, not linear-embeddable. This example appears in Proposition~3 of \cite{WuBryantEtal19}.

To see that $(X,\delta)$ is submodular, observe that for all $x \in X$,  $A \subseteq X \setminus \{x\}$ we have $\delta_x(A) = \min\{|A|,2\}$. The function $g(A) = |A|$ is strongly submodular, and therefore submodular, and so $\delta_x$ is submodular by Proposition~\ref{prop:basic_submodular}. 
\end{proof}

The extensive literature on submodular functions provides a rich source for new diversities. The following proposition can be used to construct submodular diversities from submodular functions in the same way that Proposition~\ref{prop:extend_strong} constructs negative-type diversities from strongly submodular functions. 

\begin{proposition} \label{prop:extend_sub}
\begin{enumerate}
\item Let $f:2^X \rightarrow \Re$ be a non-decreasing submodular function. Then $(X,\delta)$ given by $\delta(\emptyset) = 0$ and 
\[\delta(A) = f(A) - \min\{f(\{a\}): a \in A \}\]
is a submodular diversity.
\item Fix $x \in X$. Let $g:2^{X \setminus \{x\}} \rightarrow \Re$ be a non-decreasing submodular function with $g(\emptyset) = 0$. Then $(X,\whd)$ given by 
\[
\whd(A) = \begin{cases} g(A \setminus \{x\}) & \mbox{ if $x \in A$;} \\
g(A) - \min\{g(\{a\}):a \in A\} & \mbox{ if $A \neq \emptyset$ and $x \not \in A$;} \\
0 & \mbox{ if $A = \emptyset$,} \end{cases}
\]
is a submodular diversity.
\end{enumerate}
\end{proposition}
\begin{proof}
The proof of both claims is almost identical to that of Proposition~\ref{prop:extend_strong}.  However it is short, so we include the proof. \\
1.  Define $\rho:2^X \rightarrow \Re$ by $\rho(A) = \max\{-f(\{a\}): a \in A\}$ for all nonempty $A \subseteq X$, with $\rho(\emptyset) = \min\{-f(\{a\}):a \in  X\}$. Then $\rho$ is strongly submodular, by Proposition~\ref{prop:basic_strong}(x), and hence $h = f + \rho$ is non-decreasing and submodular. Furthermore $\delta(A) = f(A) + \rho(A)$ for all $A \neq \emptyset$. Hence for all $A,B$ such that $A \cap B \neq \emptyset$, 
\[\delta(A \cap B) + \delta(A \cup B) = h(A \cap B) + h(A \cup B) \leq h(A) + h(B) = \delta(A) + \delta(B)\]
and since $\delta(A) = 0$ when $|A| \leq 1$ the pair $(X,\delta)$ is a submodular diversity.\\
2. Define $f:2^X \rightarrow \Re$ by $f(A) = g(A \setminus \{x\})$ for all $A \subseteq X$. Then $f$ is submodular by Proposition~\ref{prop:basic_submodular} and the result follows by applying the construction in part 1. to $f$.
\end{proof}

As a direct consequence, we can define a submodular `entropy' diversity:

\begin{example}
The entropy of a set $\mathcal{A} = \{X_1,X_2,\ldots,X_k\}$ of discrete random variables  with state spaces $\Omega_1,\ldots,\Omega_k$ is defined
\[h(\{X_1,\ldots,X_k\}) = -\sum_{x_1 \in \Omega_1} \cdots \sum_{x_k \in \Omega_k} P[X_1 = x_1,\ldots,X_k = x_k] \log P[X_1 = x_1,\ldots,X_k = x_k] .\]
Let $\mathcal{X}$ be a finite set of random variables and define $(\mathcal{X},\delta)$ by $\delta(\emptyset) = 0$ and
\[\delta(\mathcal{A}) = h(\mathcal{A}) - \min\{h(\{a\}): a \in \mathcal{A}\}\]
for nonempty $\mathcal{A} \subseteq \mathcal{X}$. The entropy function is non-decreasing and submodular \cite{Fujishige1978}, so $(\mathcal{X},\delta)$ is a submodular diversity.
\end{example}

The main result of this section is a proof that submodular diversities are sublinear-embeddable, a result which then gives a novel geometric representation for submodular functions. By Theorem~\ref{thm:sublinearChar} a diversity is sublinear-embeddable if it can be written as the maximum of diversities of negative type. Equivalently, given a finite submodular diversity $(X,\delta)$ and each subset $T \subseteq X$ there is a diversity $(X,\eta)$ of negative type such that $\eta(A) \leq \delta(A)$ for all $A \subseteq X$ and $\eta(T) = \delta(T)$.  The main tools we use are Proposition~\ref{prop:extend_strong} combined with Theorem~\ref{thm:submodular_match}. While Theorem~\ref{thm:submodular_match} looks like a result from polymatroid theory, and submodular functions are the pointwise maxima of additive functions, the theorem adds a condition on singletons which forces a completely different proof strategy.

 \begin{theorem} \label{thm:submodular_match}
 Let $X$ be a finite set and $f:2^X \rightarrow \Re$ a non-decreasing and  submodular function with $f(\emptyset) = 0$. For all  $Y \subseteq X$ there is a strongly submodular function $g:2^X \rightarrow \Re$ such that $g(A) \leq f(A)$ for all $A \subseteq X$ and $g(A) = f(A)$ when $A =Y$ or $|A| \leq 1$. 
 \end{theorem}
 \begin{proof}
 We prove this by induction on $|Y|$.  When $|Y| = 0$, define $g$ by $g(\emptyset) = 0$ and $g(A)  = \max\{f(\{a\}): a \in A\}$ for nonempty $A \subseteq X$. Then $g$ is strongly submodular (Proposition~\ref{prop:basic_strong} (ix)),  $g(A) \leq f(A)$ for all $A \subseteq X$ and  $g(A) = f(A)$ when $A =Y$ or $|A| \leq 1$. \\
  
 Suppose that $k \geq 1$ and the induction hypothesis holds for when $|Y| = k-1$. Select $y \in Y$ and define $f_y:2^{X \setminus \{y\}} \rightarrow \Re$ by $f_y(A) = f(A \cup \{y\}) - f(\{y\})$ for all $A \subseteq X \setminus \{y\}$. By the induction hypothesis there is strongly submodular $g_y:2^{X \setminus \{y\}} \rightarrow \Re$ such that $g_y(A) \leq f_y(A)$ for all $A \subseteq X \setminus \{y\}$ and $g_y(A) = f_y(A)$ when $A =  Y \setminus \{y\}$ or when $|A| \leq 1$. 
 
Define $r: 2^X \rightarrow \Re$ by $r(\emptyset) = 0$ and 
\[r(A) = \max \{ f(\{a\}) + f(\{y\}) - f(\{a,y\}) :  a \in A\}\]
for nonempty $A \subseteq X$, noting that $r(A) \leq f(\{y\})$ for all $A$ and $r(A) = f(\{y\})$ when $y \in A$. Since $f$ is submodular and $f(\emptyset) = 0$ we have $f(\{a\}) + f(\{y\}) - f(\{a,y\}) \geq 0$ for all $a$.  We define $g:2^X \rightarrow \Re$ by $g(\emptyset) = 0$
 \[g(A) = r(A) + g_y(A \setminus \{y\}).\]
 By Proposition~\ref{prop:basic_strong} (x), (v) and (ii), $g$ is strongly submodular.
  
We show by cases that $g(A) \leq f(A)$ for all $A \subseteq X$ and $g(A) = f(A)$ when $A =Y$ or $|A| \leq 1$.  
For all nonempty $A \subseteq X$ such that $y \not \in A$  we have
 \begin{align*}
 g(A) & = g_y(A) + r(A) \\
 & \leq f_y(A) + \big(f(\{a\}) + f(\{y\}) - f(\{a,y\}) \big)& \mbox{ for some $a \in A$}\\
  & = f(A \cup \{y\}) + f(\{a\})  - f(\{a,y\}) \\
  & \leq f(A)
  \end{align*}
  by the submodularity inequality applied to $A$ and $\{a,y\}$. 
  
  If $y \in A$ then
  \begin{align*}
  g(A) & = g_y(A \setminus \{y\}) + r(A) \\
  & \leq f_y(A \setminus \{y\}) + f(\{y\}) \\
  & = f(A).
  \end{align*}
  If $A = Y$ then 
  \begin{align*}
  g(A) & = g_y(Y \setminus \{y\}) + r(Y) \\
  & = f_y(Y \setminus \{y\}) + f(\{y\})\\
  & = f(Y).
  \end{align*}
  For all $a \neq y$,
   \[g(\{a\}) = g_y(\{a\}) + \left( f(\{a\}) + f(\{y\}) - f(\{a,y\})  \right) = f(\{a\})\]
  while
   \[g(\{y\}) = g_y(\emptyset) + f(\{y\}) = f(\{y\})\]
and
\[g(\emptyset) =  g_y(\emptyset) + r(\emptyset) = 0.\]  
 
  By induction, the result holds for all $Y$. 
  \end{proof}

We can now prove the main embedding result. 

\begin{theorem} \label{thm:embed_submodular}
Let $(X,\delta)$ be a finite, submodular diversity. Then $(X,\delta)$ is sublinear-embeddable.
\end{theorem}
\begin{proof}
For each $T \subseteq X$ we will construct a negative-type diversity $(X,\eta)$ such that $\eta(A) \leq \delta(A)$ for all $A \subseteq X$ and $\eta(T) = \delta(T)$. This shows that $(X,\delta)$ is the maximum of a finite set of negative-type diversities, so is sublinear-embeddable by Theorem~\ref{thm:sublinearChar}. \\

First note that if $\delta(T) = 0$ then we can use $\eta = 0$. We therefore assume that $\delta(T) > 0$ (and hence $|T| > 1$).  

Fix $T \subseteq X$ and $x \in T$. As $(X,\delta)$ is submodular, the function $\delta_x:2^{X \setminus \{x\}} \rightarrow \Re$ given by $\delta_x(A) = \delta(A \cup \{x\})$ is submodular, non-decreasing, and $\delta_x(\emptyset) = 0$. By Theorem~\ref{thm:submodular_match} there is a strongly submodular function $g:2^{X \setminus \{x\}} \rightarrow \Re$ such that $g(A) \leq \delta_x(A)$ for all $A \subseteq X \setminus \{x\}$ and $g(A) = \delta_x(A)$ when $A = T \setminus \{x\}$ or $|A| \leq 1$. 

Define $(X,\eta)$ by $\eta(\emptyset)  = \eta(\{x\}) = 0$ and 
\begin{align*}
\eta(A) &= g(A) - \min\{g(\{a\}):a \in A\} \\
\eta(A \cup \{x\}) & = g(A) 
\end{align*}
for all nonempty $A \subseteq X \setminus \{x\}$.  Then $(X,\eta)$ is negative type, by Proposition~\ref{prop:extend_strong} and $\eta(T) = g(T \setminus \{x\}) = \delta(T)$. For all $A \subseteq X \setminus \{x\}$, 
\begin{equation}
\eta(A \cup \{x\}) = g(A) \leq \delta(A \cup \{x\}) \label{eq:eta_matches}
\end{equation}
and 
\begin{align*}
\eta(A) & = g(A) - \min\{g(\{a\}):a \in A\} \\
& \leq \delta(A \cup \{x\}) - \min\{\eta(\{a,x\}) : a \in A \} \\
& = \delta(A \cup \{x\}) - \eta(\{a^*,x\}) & \mbox{ for minimizer $a^* \in A$ } \\
& = \delta(A \cup \{x\}) - \delta(\{a^*,x\}) \\
& \leq \delta(A),
\end{align*}
by the diversity triangle inequality (D2).
\end{proof}

Note that the converse of Theorem~\ref{thm:embed_submodular} is false: there are examples of diversities  which are sublinear-embeddable but not submodular.

\begin{example} \label{ex:sublinear_not_submodular1}
Let $X = \{a,b,c,d\}$. The diversity $(X,\delta)$ with $\delta$ given by
\[\delta(A) = \begin{cases} 3 & \mbox{if $|A| = 4$;}\\ 2 & \mbox{if $2 \leq |A| \leq 3$} \\ 0 & \mbox{ otherwise} \end{cases} \]
is sublinear-embeddable but not submodular.
\end{example}
\begin{proof}
Clearly $(X,\delta)$ is a symmetric diversity. It satisfies 
\[\frac{\delta(A \setminus \{a\}) }{\delta(A)} \geq \frac{|A|-2}{|A|-1}\]
for all $A \subseteq X$ such that $|A| \geq 2$ , $a \in A$. By Theorem~4.3 of \cite{bryant2023diversities}, $(X,\delta)$ is Minkowski-embeddable, so by Theorem~\ref{thm:sublinearChar} it is sublinear-embeddable. It is not, however, submodular, as can be seen by letting $A = \{a,b,c\}$ and $B = \{b,c,d\}$ and observing
\[\delta(A \cup B) + \delta(A \cap B) = 5 > 4 = \delta(A) + \delta(B).\]
\end{proof}

\begin{example} \label{ex:sublinear_not_submodular2}
Let $(\Re^2,\delta_K)$ be the Minkowski diversity where $K$ is the unit ball. Hence for any finite, nonempty, set of points $A \subseteq \Re^2$, $\delta_K(A)$ equals the circumradius of $A$.  Let $X = \{(-2,0),(0,1),(0,-1),(2,0)\}$. We see from Figure~\ref{fig:circum_not_submod} that $(X,\delta_K)$, the restriction  of $(\Re^2,\delta_K)$ to $X$, is sublinear-embeddable (by construction) but is not submodular. 
\end{example}

\begin{figure}[htb]
\centerline{\includegraphics[width=0.5\textwidth]{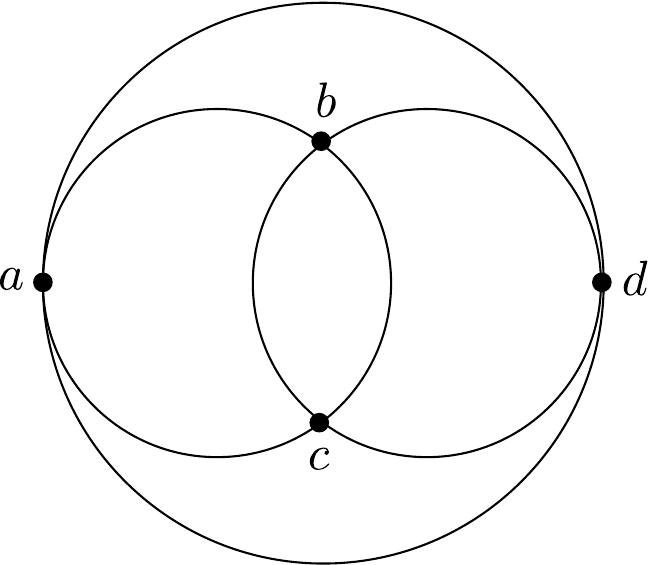}}
\caption{\label{fig:circum_not_submod} An example showing that not every sublinear-embeddable diversity is submodular. Here $X = \{a,b,c,d\}$ are the points $(-2,0)$, $(0,1)$, $(0,-1)$ and $(2,0)$. Let $A = \{a,b,c\}$ and $B = \{b,c,d\}$. Then $\delta_K(A \cup B) = 2$, $\delta_K(A \cap B) = 1$ and $\delta_K(A) =\delta_K(B) = 5/4$. }
\end{figure}

\section{Sublinear diversities and XOS functions}

We have shown that submodular diversities are sublinear-embeddable. Since (by Theorem~\ref{thm:sublinearChar}) the sublinear-embeddable diversities are exactly those given by the maxima of negative-type diversities, it follows that the {\em maximum} of two or more submodular diversities is also sublinear-embeddable. In this case, the converse is also true.

\begin{theorem} \label{thm:max_submod}
A finite diversity $(X,\delta)$ is sublinear-embeddable if and only if there are submodular diversities $(X,\delta_1),(X,\delta_2),\ldots,(X,\delta_m)$ such that 
\[\delta(A) = \max\{\delta_1(A),\ldots,\delta_m(A)\}\]
for all $A \subseteq X$. 
\end{theorem}
\begin{proof}
By Theorem~\ref{thm:sublinearChar} there is a finite collection $(X,\delta_1),\ldots,(X,\delta_m)$ of negative type diversities such that $\delta(A) = \max\{\delta_i(A):i=1,\ldots,m\}$ for all $A \subseteq X$. By Proposition~\ref{prop:basic_submodular_diversities}, each of these is submodular. 

Conversely, by Theorem~\ref{thm:embed_submodular} each of the diversities $(X,\delta_i)$ is sublinear-embeddable, so their maximum is also sublinear-embeddable.
\end{proof}

A function $f:2^X \rightarrow \Re$ is {\em additive} if there are non-negative weights $\{w_x\}_{x \in X}$ such that $f(\emptyset) = 0$ and $f(A) = \sum_{a \in A} w_a$ for all nonempty $A \subseteq X$. We say that $f:2^X \rightarrow \Re$ is an {\em XOS function} if there are additive functions $f_1,\ldots,f_m$ on $2^X$ such that for all $A \subseteq X$, $f(A) = \max\{f_i(A):i=1,\ldots,m\}$ \cite{nisan2000bidding}. XOS functions are subadditive. They also generalize non-decreasing submodular functions (with $f(\emptyset) = 0$). 

\begin{proposition} \label{prop:XOS}
A function $f:2^X \rightarrow \Re$ is XOS if and only if there are non-decreasing submodular functions $f_1,\ldots,f_m$ such that $f_i(\emptyset) = 0$ for all $i$ and 
\[f(A) = \max\{f_1(A),\ldots,f_m(A)\}\]
for all $A \subseteq X$.
\end{proposition}
\begin{proof}
Additive functions are non-decreasing and strongly submodular (Proposition~\ref{prop:basic_strong} (vi) ) and therefore submodular, so one direction follows straight from the definition. 

For the converse, suppose that $f_1,\ldots,f_m$ are non-decreasing, submodular functions with $f_i(\emptyset) = 0$ for all $i$. For each $f_i$ define the {\em polymatroid}
\[P_{f_i} = \left\{x \in \Re_{\geq 0}^X: \sum_{a \in A} x_a  \leq f_i(A) \mbox{ for all $A \subseteq X$ } \right\}.\]
Note that $P_{f_i}$ is bounded, since for each $a \in X$, $0 \leq x_a \leq f_i(\{a\})$, for each $x \in P_{f_i}$.
Then for all $A \subseteq X$, $f_i(A) = \max \{ \sum_{a \in A} x_a : x \in P_{f_i} \}$, see \cite{Lovasz1983}, or Corollary~44.3g in \cite{Schrijver2003}. Since $P_{f_i}$ is a bounded polytope with finitely many vertices, and each maximum is attained at a vertex, $f_i$ is the maximum of a finite number of additive functions. This implies that $f$ is too.
\end{proof}

A pair $(X,\delta)$ with $\delta(A) = 0$ when $|A| \leq 1$ is a diversity when $\delta_x$ is non-decreasing and subadditive for all $x \in X$; it is a submodular diversity when $\delta_x$ is non-decreasing and submodular for all $x \in X$,  and negative type when $\delta_x$ is strongly submodular for all $x \in X$. Continuing along the same lines, we say that $(X,\delta)$ is an {\em XOS diversity} if for all $x \in X$ the function $\delta_x$ is  XOS. It follows from Theorem~\ref{thm:max_submod} and Proposition~\ref{prop:XOS} that every finite sublinear-embeddable diversity is an XOS diversity. We might expect the converse to hold. An AI-assisted search through small $n$ cases turned up the following counterexample. 

\begin{example} \label{ex:XOS}
Let $X = \{1,2,3,4,5\}$ and define $\delta$ by 
\[\delta(A) = \begin{cases} 0 & \mbox{ if $|A| \leq 1$ }\\
1 & \mbox{ if $|A| = 2$ or $A$ is one of $\{1,2,4\},\{2,3,5\},\{1,3,4\},\{2,4,5\},\{1,3,5\}$ } \\
2 & \mbox{ otherwise.}  \end{cases} \]
Then $(X,\delta)$ is an XOS diversity.
\end{example}
\begin{proof}
By rotational symmetry, if $\delta_x$ is XOS for $x = 5$ then it is XOS for all $x \in X$. Note that $\delta_5(\emptyset) = 0$, $\delta_5(A) = 1$ on singletons and pairs $\{1,3\},\{2,3\},\{2,4\}$ and $\delta_5(A) = 2$ otherwise. Define
\[W = \left[ \begin{matrix} %
1 & 0 & 0 & 0 & 1 & 1 & 0\\
0 & 1 & 0 & 0 & 0 & 1 & 0\\
0 & 0 & 1 & 0 & 0 & 0 & 1\\
0 & 0 & 0 & 1 & 1 & 0 & 1
\end{matrix} \right].\]
Then for nonempty $A \subseteq \{1,2,3,4\}$, $\delta_5(A) = \max_j \sum_{i \in A} W_{ij}$. Hence $\delta_5$ is XOS, and $(X,\delta)$ is an XOS diversity.
\end{proof}

\begin{example}
The diversity defined in Example~\ref{ex:XOS} is not the maximum of submodular diversities.
\end{example}
\begin{proof}
Let $(X,\eta)$ be a submodular diversity such that $\eta(A) \leq \delta(A)$ for all $A \subseteq X$. By submodularity, monotonicity and (D4) we have 
\begin{equation*}
\begin{array}{rcl}
\eta(\{2,3,5\}) + \eta(\{2,4,5\}) &\geq \eta(\{2,3,4,5\}) + \eta(\{2,5\}) & \geq \eta(\{3,4,5\}) + \eta(\{2,5\}) \\
\eta(\{1,3,5\}) + \eta(\{1,3,4\}) &\geq \eta(\{1,3,4,5\}) + \eta(\{1,3\}) & \geq \eta(\{3,4,5\}) + \eta(\{1,3\}) \\
\eta(\{2,5\}) + \eta(\{1,3\}) + \eta(\{1,2,4\}) & \geq \eta(\{1,2,3,4,5\}) & \geq \eta(\{3,4,5\}).
\end{array}
\end{equation*}
Combining these we have
\begin{align*}
3 \eta(\{3,4,5\}) & \leq \eta(\{2,3,5\}) + \eta(\{2,4,5\}) + \eta(\{1,3,5\}) + \eta(\{1,3,4\}) + \eta(\{1,2,4\}) \\
& \leq  \delta(\{2,3,5\}) + \delta(\{2,4,5\}) + \delta(\{1,3,5\}) + \delta(\{1,3,4\}) + \delta(\{1,2,4\}) \\
& \leq 5.
\end{align*}
Hence $\eta(\{3,4,5\}) < 2 = \delta(\{3,4,5\})$. There is no submodular diversity $\eta$ with $\eta(A) \leq \delta(A)$ for all $A$ and $\eta(\{3,4,5\}) = \delta(\{3,4,5\})$, and $\delta$ is not a maximum of submodular diversities.
\end{proof}

In summary, we have the following chain of (strict) inclusions for finite diversities:
\begin{align*}
\mbox{phylogenetic} & \subsetneq \mbox{ $\ell_1$-embeddable} \\
& \subsetneq \mbox{ negative type $=$ linear-embeddable}\\
& \subsetneq \mbox{ submodular }\\
& \subsetneq \mbox{ Minkowski-embeddable $=$ sublinear-embeddable} \\
& \subsetneq \mbox{ XOS}.
\end{align*}
All inclusions are strict.

\section{Representing submodular functions}

In this final section we take what we have learnt about submodular diversities (and their close relatives) and map it back to give new results about submodular functions (and their close relatives). Specifically, we explore the implications of the geometric embeddings for functions in these classes. 

From Proposition~\ref{prop:neg_is_strongsubmod} we see that a diversity $(X,\delta)$ has negative type if each of the local functions $\delta_x$ is strongly submodular. The fact that this holds simultaneously for all $x$ places constraints on the relationships between, say, the functions $\delta_x$ and $\delta_y$ for $x \neq y$. However, as we saw in Proposition~\ref{prop:extend_strong} it does not place constraints on which strongly submodular functions equal $\delta_x$ for some negative type diversity $(X,\delta)$. This allows us to map our embedding results for diversities of negative type to general embedding results for strongly submodular functions. The same applies for submodular diversities and submodular functions and, with a catch, for XOS functions. 

\begin{theorem} \label{thm:embed_strong}
Let $X$ be a finite set and let $f:2^X \rightarrow \Re$ be a set function with $f(\emptyset) = 0$. Then $f$ is strongly submodular if and only if there is $k \geq 1$, a map $\phi:X \rightarrow  \Re^k$ and a $k$-dimensional simplex $K \subseteq \Re^k$ such that for all $A \subseteq X$,
\[f(A) = R(\phi(A) \cup \{0\},K) = \inf \{\lambda \geq 0: \phi(A) \cup \{0\} \subseteq \lambda K + z \mbox{ for some $z \in \Re^k $} \}.\]
\end{theorem}
\begin{proof}
Suppose that $f$ is strongly submodular and $z \not \in X$. From Proposition~\ref{prop:extend_strong} there is a diversity $(X \cup \{z\},\delta)$ of negative type such that $f(A) = \delta(A\cup \{z\})$ for all $A \subseteq X$. By Theorem~\ref{thm:linearChar} there is an embedding from $X \cup \{z\}$ into $\Re^k$ for some $k$ such that $\delta(A) = \delta_K(\phi(A))$ for a Minkowski diversity $\delta_K$ with kernel equal to the simplex. As $\delta_K$ is translation invariant we can assume that $\phi(z) = 0$, and the result follows directly.

The converse follows directly from Theorem~\ref{thm:linearChar}  and Proposition~\ref{prop:neg_is_strongsubmod}.
\end{proof}

\begin{example} \label{ex:simpleK}
As a simple illustration of the theorem, suppose that $w_x \geq 0$ for all $x \in X$ and $f(A) = \max\{w_a: a \in A\}$ with $f(\emptyset) = 0$. Then $f$ is strongly submodular. Let $\{\mathbf{e}_x\}_{x \in X}$ be the standard basis for $\Re^X$, let $\phi:X \rightarrow \Re^{X}$ be given by $\phi(x) = w_x \mathbf{e}_x$ for all $x \in X$ and let $K$ be the simplex with vertices $\mathbf{0}$ and $\{\mathbf{e}_x : x \in X\}$. Then $\delta_K(\phi(A) \cup \{0\}) = f(A)$ for all $A \subseteq X$. The case of $|X| = 2$ is given in Figure~\ref{fig:simpleK}. 
\end{example}

\begin{figure}[htb]
\centerline{\includegraphics[width=0.5\textwidth]{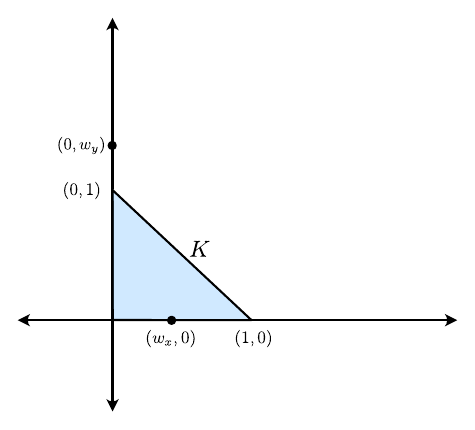}}
\caption{\small \label{fig:simpleK} An illustration of Example~\ref{ex:simpleK} for the case that $X = \{x,y\}$. }
\end{figure}

From Theorem~\ref{thm:embed_strong} we see that every strongly submodular function embeds into $(\Re^k,\delta_K)$ with $K$ given by a $k$-dimensional simplex. Any $k$-dimensional simplex will do: we can apply any invertible affine map to both $K$ and $\phi$ to achieve the same result. 

More generally, we can use the same process to embed submodular functions. 

\begin{proposition} \label{prop:embed_submod}
Let $X$ be a finite set and let $f:2^X \rightarrow \Re$ be a non-decreasing submodular function such that $f(\emptyset) = 0$. There is $k \geq 1$, a map $\phi:X \rightarrow \Re^k$ and a compact convex set $K \subseteq \Re^k$ with nonempty interior such that, for all $A \subseteq X$
\[f(A) = R(\phi(A) \cup \{0\},K) = \inf \{\lambda \geq 0: \phi(A) \cup \{0\} \subseteq \lambda K + z \mbox{ for some $z \in \Re^k $} \}.\]
\end{proposition}
\begin{proof}
Suppose that $x \not \in X$. By Proposition~\ref{prop:extend_sub} there is a submodular diversity $(X \cup \{x\},\delta)$ such that $\delta_x = f$. From Theorem~\ref{thm:embed_submodular} there is $k \geq 1$, a map $\phi: X \cup \{x\} \rightarrow \Re^k$ and a convex set $K$ such that $\delta_K(\phi(A)) = \delta(A)$ for all $A \subseteq X \cup \{x\}$. As $\delta_K$ is translation invariant (Proposition~1 of \cite{bryant2024linear}), we can assume that $\phi(x) = \mathbf{0}$. Hence for all $A \subseteq X$ we have
\[f(A) = \delta_x(A) = \delta(A \cup \{x\}) = \delta_K(\phi(A) \cup \{\mathbf{0}\}),\]
as required.
\end{proof}

Proposition \ref{prop:embed_submod} shows that every non-decreasing submodular function $f$ with $f(\emptyset) = 0$ has a representation via the generalized circumradius. It does not state {\em which} functions can have the same representation. The following theorem closes that gap.

\begin{theorem}
    Let $X$ be a finite set and let $f:2^X \rightarrow \Re$ be a non-decreasing function such that $f(\emptyset) = 0$. Then $f$ is XOS if and only if there is $k \geq 1$, a map $\phi:X \rightarrow \Re^k$ and a compact convex set $K \subseteq \Re^k$ with nonempty interior such that, for all $A \subseteq X$
\[f(A) = R(\phi(A) \cup \{0\},K) = \inf \{\lambda \geq 0: \phi(A) \cup \{0\} \subseteq \lambda K + z \mbox{ for some $z \in \Re^k $} \}.\]
\end{theorem}
\begin{proof}
    Suppose that $f:2^X \rightarrow \Re$ is XOS. By Proposition~\ref{prop:XOS} there are submodular non-decreasing functions $f_1,\ldots,f_m$ such that $f_i(\emptyset) = 0$ for all $i$ and 
    \[f(A) = \max\{f_1(A),\ldots,f_m(A)\}\]
    for all $A \subseteq X$. Applying Proposition~\ref{prop:extend_sub} to each of these we obtain, for some $x \not \in  X$, submodular diversities $(X \cup \{x\},\delta^{(1)}),\ldots,(X \cup \{x\},\delta^{(m)})$  such that $\delta^{(i)}_x = f_i$ for $i=1,\ldots,m$. Let $(X\cup \{x\},\whd)$ be the diversity given by 
    \[\whd(A) = \max\{\delta^{(i)}(A): i=1,\ldots,m\}\]
    for all $A \subseteq X \cup \{x\}$. By construction 
    \[f(A) = \whd_x(A) = \whd(A \cup \{x\})\]
    for all $A \subseteq X$.
    
    By Theorem~\ref{thm:max_submod}, $(X \cup \{x\},\whd)$ is sublinear-embeddable, so there is $k \geq 1$, a map $\phi:X \cup \{x\} \rightarrow \Re^k$ and a compact, convex set $K$ with nonempty interior such that 
    \[\whd(A) = \delta_K(\phi(A))\]
    for all $A \subseteq X \cup \{x\}$. As $\delta_K$ is translation invariant, we can assume $\phi(x) = \mathbf{0}$. Then for all $A \subseteq X$,
    \[f(A) = \whd(A \cup \{x\}) = \delta_K(\phi(A) \cup \{\mathbf{0}\}).\]

    For the converse, suppose that for all $A \subseteq X$,
\[f(A) =  \inf \{\lambda \geq 0: \phi(A) \cup \{0\} \subseteq \lambda K + z \mbox{ for some $z \in \Re^k $} \}.\]
Choose $x \not \in X$, define $\phi(x) = \mathbf{0}$, and the diversity $(X \cup \{x\},\delta)$ by $\delta(A) = \delta_K(\phi(A))$ for all $A \subseteq X \cup \{x\}$. By Theorem~\ref{thm:max_submod}, $\delta$ is the maximum of submodular diversities $(X \cup \{x\},\delta^{(1)}),\ldots,(X \cup \{x\},\delta^{(m)})$ so that by Proposition~\ref{prop:XOS}, $f = \delta_x = \max\{\delta^{(i)}_x\}$ is XOS.     
\end{proof}

 \section*{Acknowledgements}
 
PT was supported by a Natural Sciences and Engineering Research Council (Canada) Discovery Grant (RGPIN-2025-04769). \\
 
Claude and ChatGPT were used to help identify relevant literature, to test conjectures, to help check proofs, to suggest ideas for proofs (both fruitful and not) and to format the bibliography. All text and proofs were written by the authors.

	\bibliographystyle{plain}
	\bibliography{submodular_diversity}

 \end{document}